\documentclass[12pt]{amsart}

\usepackage{amssymb,latexsym}
\usepackage{verbatim,euscript,array}
\usepackage{fullpage,color}
\usepackage{tikz}
\usetikzlibrary{arrows,shapes,trees}

\usepackage[colorlinks=true,linkcolor=blue,urlcolor=my_color,citecolor=redi]{hyperref}

\definecolor{my_color}{rgb}{0,0.5,0.5}
\definecolor{MIXT}{rgb}{0.4,0.3,0.6}
\definecolor{mixt}{rgb}{0.5,0.3,0.2}
\definecolor{sin}{rgb}{0,0.5,0.5}
\definecolor{darkblue}{rgb}{0,0.1,0.8}
\definecolor{redi}{rgb}{0.6,0,0.4}

\input {cyracc.def}
\numberwithin{equation}{section}

\font\tencyr=wncyr10 
\def\rus{\tencyr\cyracc}

\newtheorem{thm}{Theorem}[section]
\newtheorem{lm}[thm]{Lemma}
\newtheorem{cl}[thm]{Corollary}
\newtheorem{prop}[thm]{Proposition}

\theoremstyle{remark}
\newtheorem{rmk}[thm]{Remark}

\theoremstyle{definition}
\newtheorem{ex}[thm]{Example}
\newtheorem{df}{Definition}
\newtheorem*{rema}{Remark}

\newcommand{\eus}{\EuScript}
\newcommand {\ah}{{\mathfrak a}}
\newcommand {\be}{{\mathfrak b}}

\newcommand {\g}{{\mathfrak g}}

\newcommand {\el}{{\mathfrak l}}

\newcommand {\te}{{\mathfrak t}}
\newcommand {\ut}{{\mathfrak u}}

\newcommand {\sln}{{\mathfrak{sl}}_n}

\newcommand {\gA}{{\eus A}}

\newcommand {\gN}{{\eus N}}
\newcommand {\gH}{{\eus H}}

\newcommand {\esi}{\varepsilon}
\newcommand {\ap}{\alpha}

\newcommand {\lb}{\lambda}

\newcommand {\HW}{\widehat W}

\newcommand {\co}{{\mathcal O}}

\newcommand {\BR}{{\mathbb R}}

\newcommand {\BZ}{{\mathbb Z}}
\newcommand {\BC}{{\mathbb C}}

\newcommand {\hot}{{\mathsf{ht}}}

\newcommand {\supp}{{\mathsf{supp}}}

\newcommand {\GR}[2]{{\textrm{{\sf\bfseries #1}}}_{#2}}

\newcommand {\Ab}{\mathfrak{Ab}}
\newcommand {\Abo}{\mathfrak{Ab}^o}
\newcommand {\AD}{\mathfrak{Ad}}

\newcommand {\thi}{{\lfloor\theta/2\rfloor}}
\newcommand {\tthe}{{\lceil \theta/2\rceil}}
\newcommand {\bap}{\boldsymbol{\hat\ap}}

\newcommand {\beq}{\begin{equation}}
\newcommand {\eeq}{\end{equation}}

\newcommand{\curge}{\succcurlyeq}
\newcommand{\curle}{\preccurlyeq}
\renewcommand{\le}{\leqslant}
\renewcommand{\ge}{\geqslant}

\newenvironment{E8}[8]{
{\footnotesize\begin{tabular}{@{}c@{}}
{#1}{#2}{#3}{#4}\lower2.3ex\vbox{\hbox{{#5}\rule{0ex}{0.1ex}}
\hbox{\hspace{0.3ex}\rule{0ex}{.1ex}\rule{0ex}{.1ex}}\hbox{{#8}\strut}}{#6}{#7}
\end{tabular}}}

\begin{document}
\setlength{\parskip}{2pt plus 4pt minus 0pt}
\hfill {\scriptsize \today} 
\vskip1.5ex

\title[Abelian ideals and root systems]%
{Abelian ideals of a Borel subalgebra and root systems, II}
\author{Dmitri I. Panyushev}
\address[]{Institute for Information Transmission Problems of the R.A.S., Bolshoi Karetnyi per. 19,  
127051 Moscow,  Russia}
\email{panyushev@iitp.ru}
\thanks{This research is partially supported by the R.F.B.R. grant {\rus N0} 16-01-00818.}

\begin{abstract}
Let $\g$ be a simple Lie algebra with a Borel subalgebra $\be$ and $\Ab$ the set of abelian ideals of $\be$. Let $\Delta^+$ be the corresponding set of positive roots. 
We continue our study of combinatorial properties of the partition of $\Ab$ parameterised by the long positive roots. 
In particular, the union of an arbitrary set of maximal abelian ideals is described, if $\g\ne\sln$.
We also characterise the greatest lower bound of two positive roots, when it 
exists, and point out interesting subposets of $\Delta^+$ that are modular lattices. 
\end{abstract}

\maketitle

\section*{Introduction}

\noindent
Let $\g$ be a simple Lie algebra over $\BC$, with a  triangular decomposition 
$\g=\ut\oplus\te\oplus \ut^-$. Here $\te$ is a Cartan and $\be=\ut\oplus\te$ is a fixed Borel 
subalgebra.  The theory of abelian ideals of $\be$ is based on their relationship, due to D.~Peterson, with 
the {\it minuscule elements\/} of the affine Weyl group $\HW$ (see Kostant's account in~\cite{ko98}; another approach is presented in \cite{cp1}). 
In this note, we elaborate on some topics related to the combinatorial theory of abelian ideals, which can be regarded as a sequel to \cite{jems}. We mostly work in the combinatorial setting, i.e., the abelian ideals of $\be$, which are sums of 
root spaces of $\ut$, are identified with the corresponding sets of positive roots. 

Let $\Delta$ be the root system of $(\g,\te)$ in the vector space $V=\te_\BR^*$, $\Delta^+$ the set of 
positive roots in $\Delta$ corresponding to $\ut$,  $\Pi$ the set of simple roots in $\Delta^+$, and $\theta$  
the {highest root} in  $\Delta^+$. Then $W$ is the Weyl group and $(\ ,\ )$ is a $W$-invariant scalar 
product on $V$. We equip  $\Delta^+$  with the usual partial ordering `$\curge$'. An {\it upper ideal\/} (or 
just an {\it ideal}) of $(\Delta^+,\curge)$ is a subset $I\subset \Delta^+$ such that if
$\gamma\in I, \nu\in\Delta^+$, 
and $\nu+\gamma\in\Delta^+$, then $\nu+\gamma\in I$. An upper ideal $I$ is {\it abelian}, if
$\gamma'+\gamma''\not\in \Delta^+$ for all $\gamma',\gamma''\in I$. The set of minimal elements of $I$ is 
denoted by $\min(I)$. It also makes sense to consider the maximal elements of the complement of $I$, 
denoted $\max(\Delta^+\setminus I)$. 

Write $\Ab$ (resp. $\AD$) for the set of all abelian (resp. all upper) ideals of $\Delta^+$ and 
think of them as posets with respect to inclusion. The upper ideal {\it generated by\/} $\gamma$ is
$I\langle{\curge}\gamma\rangle=\{\nu\in\Delta^+\mid \nu\curge \gamma\}$. Then
$\min (I\langle{\curge}\gamma\rangle)=\{\gamma\}$. A root $\gamma\in\Delta^+$ is said to be {\it commutative}, if $I\langle{\curge}\gamma\rangle\in\Ab$. Write $\Delta^+_{\sf com}$ for the set of all commutative roots.
This notion was introduced in \cite{jac06}, and the subset  $\Delta^+_{\sf com}$ for each $\Delta$ is explicitly described in~\cite[Theorem\,4.4]{jac06}. Note that $\Delta^+_{\sf com}\in \AD$.

Let $\Abo$ denote the set of nonempty abelian ideals and $\Delta^+_l$  the set of long positive roots.  
In~\cite[Sect.\,2]{imrn}, we defined a mapping $\tau: \Abo \to \Delta^+_l$, which is onto. 
Letting $\Ab_\mu=\tau^{-1}(\mu)$, we get a partition of $\Abo$ parameterised by $\Delta^+_l$. Each 
$\Ab_\mu$ is  a subposet of $\Ab$ and, moreover, $\Ab_\mu$ has a unique minimal and unique 
maximal element (ideal)~\cite[Sect.\,3]{imrn}. These extreme abelian ideals in $\Ab_\mu$  
are denoted by $I(\mu)_{\sf min}$ and $I(\mu)_{\sf max}$. Then  
$\{I(\ap)_{\sf max}\mid \ap\in\Pi_l\}$ are exactly the maximal abelian ideals of $\be$.
 
In this article, we first establish a property of $(\Delta^+,\curge)$, which seems to be new. It was proved 
in~\cite[Appendix]{jems} that,  for any $\eta_1,\eta_2\in \Delta^+$, there exists the {\it least upper bound}, 
denoted $\eta_1\vee\eta_2$. Moreover, an explicit formula for $\eta_1\vee\eta_2$ is also given. Here 
we prove that the {\it greatest lower bound}, $\eta_1\wedge\eta_2$, exists if and only if 
$\supp(\gamma_1)\cap\supp(\gamma_2)\ne\varnothing$. Furthermore, if 
$\eta_i=\sum_{\ap\in\Pi} c_{i\ap}\ap$, then 
$\eta_1\wedge\eta_2=\sum_{\ap\in\Pi} \min\{c_{1\ap},c_{2\ap}\}\ap$.
This also implies that $I\langle{\curge}\eta\rangle$ is a modular lattice for any $\eta\in\Delta^+$, see
Theorem~\ref{thm:inf=min}. Another example a modular lattice inside $\Delta^+$ is the subposet
$\Delta_\ap(i)=\{\gamma\in\Delta^+\mid  \hot_\ap(\gamma)=i\}$, where $\ap\in\Pi$ and $\hot_\ap(\gamma)$ is the coefficient of $\ap$ in the expression of $\gamma$ via $\Pi$.

Using properties of `$\vee$' and `$\wedge$' and $\BZ$-gradings of $\g$, we prove uniformly that if 
$\Delta$ is not of type $\GR{A}{n}$, then $\Delta^+_{\sf nc}:=\Delta^+\setminus \Delta^+_{\sf com}$ has 
the unique maximal element, which is 
$\lfloor\theta/2\rfloor :=\sum_{\ap\in\Pi} \lfloor\hot_\ap(\theta)/2\rfloor \ap$, see 
Section~\ref{sect:Z-gr-&-nc}. In particular, $\lfloor\theta/2\rfloor $ is a root. (Note that if $\Delta$ is of 
type $\GR{A}{n}$, then $\lfloor\theta/2\rfloor =0$ and $\Delta^+_{\sf nc}=\varnothing$.) We also describe
the maximal abelian ideals $I(\ap)_{\sf max}$ if $\hot_\ap(\theta)$ is odd.

In Section~\ref{sect:subsets-of-Pi_l}, we study the sets of maximal and minimal elements related to
abelian ideals of the form $I(\ap)_{\sf min}$ and $I(\ap)_{\sf max}$, with $\ap\in\Pi_l:=\Pi\cap\Delta^+_l$. 
\begin{thm}    \label{thm:main2}
If $S\subset \Pi_l$ is arbitrary and $\Delta$ is not of type $\GR{A}{n}$, then there is the bijection
\[
 \eta\in \min\bigl( \bigcap_{\ap\in S}I(\ap)_{\sf min}\bigr)   \stackrel{1:1}{\longmapsto} 
 \eta'=\theta-\eta \in \max \bigl(\Delta^+\setminus \bigcup_{\ap\in S} I(\ap)_{\sf max}\bigr) .
\] 
\end{thm}
\noindent
Our proof is conceptual and relies on the fact $\theta$ is a multiple of a fundamental weight if
$\Delta$ is not of type $\GR{A}{n}$. For $\GR{A}{n}$, the same bijection holds if $S$ is a {\bf connected} subset on the Dynkin diagram. The case in which $\#S=1$ was considered earlier 
in~\cite[Theorem\,4.7]{jems}.
This has some interesting consequences if $S=\Pi_l$ and hence $\bigcup_{\ap\in \Pi_l} I(\ap)_{\max}=
\Delta^+_{\sf com}$, see Proposition~\ref{prop:ap-br}.

In Section~\ref{sect:interval}, we describe the interval $[\thi, \theta-\thi]$ inside the poset $\Delta^+$.

\section{Preliminaries}    \label{sect:1}

\noindent
We have $\Pi=\{\ap_1,\dots,\ap_n\}$, the vector space $V=\oplus_{i=1}^n{\mathbb R}\ap_i$, the  Weyl 
group $W$ generated by  simple reflections $s_\ap$ ($\ap\in\Pi$), and a $W$-invariant inner product 
$(\ ,\ )$ on $V$. Set $\rho=\frac{1}{2}\sum_{\nu\in\Delta^+}\nu$. The partial ordering `$\curle$' in 
$\Delta^+$ is defined by the rule  that $\mu\curle\nu$ if $\nu-\mu$ is a non-negative integral linear 
combination of simple roots. Write $\mu\prec\nu$, if $\mu\curle \nu$ and $\mu\ne \nu$.
If $\mu=\sum_{i=1}^n c_i\ap_i\in\Delta$, then $\hot_{\ap_i}(\mu):=c_i$, $\hot(\mu):=\sum_{i=1}^n c_i$ 
and $\supp(\mu)=\{\ap_i\in\Pi\mid  c_i\ne 0\}$.

The Heisenberg ideal $\gH:=\{\gamma\in \Delta^+ \mid (\gamma, \theta)\ne 0\}=
\{\gamma\in \Delta^+ \mid (\gamma, \theta)> 0\}\in\AD$ plays a prominent role in the theory of abelian 
ideals and posets $\Ab_\mu=\tau^{-1}(\mu)$. 
\\ \indent
Let us collect  some known results that are frequently used below.
\begin{itemize}
\item If $I\in\AD$ is not abelian, then there exist $\eta,\eta'\in I$ such that $\eta+\eta'=\theta$,
see \cite[p.\,1897]{imrn}. Therefore, $I\not\in\Ab$ if and only if $I\cap\gH\not\in\Ab$.

\item $I=I(\mu)_{\sf min}$ for some $\mu\in\Delta^+_l$  if and only if 
$I\subset  \gH$~\cite[Theorem\,4.3]{imrn};

\item $\# I(\mu)_{\sf min}=(\rho,\theta^\vee-\mu^\vee)+1$~\cite[Theorem\,4.2(4)]{imrn};

 
\item For $I\in\Abo$, we have $I\in\Ab_\mu$ if and only if $I\cap\gH=I(\mu)_{\sf min}$~\cite[Prop.\,3.2]{jems};
\item The set of (globally) maximal abelian ideals is $\{I(\ap)_{\sf max}\mid \ap\in\Pi_l\}$~\cite[Corollary\,3.8]{imrn}.
\item For any $\mu\in\Delta^+_l$, there is a unique element of minimal length in $W$ that takes $\theta$ to 
$\mu$~\cite[Theorem\,4.1]{imrn}. Writing $w_\mu$ for this element, one has 
$\ell(w_\mu)=(\rho,\theta^\vee-\mu^\vee)$~\cite[Theorem\,4.1]{imrn}.
\item Let $\gN(w)$ be the inversion set of $w\in W$. By \cite[Lemma\,1.1]{mics},
\[
   I(\mu)_{\sf min}=\{\theta\}\cup \{\theta-\gamma\mid \gamma\in \gN(w_\mu)\}.
\]
\end{itemize}

\noindent
For each $\eta\in \gH\setminus\{\theta\}$
there is a unique $\eta'\in \gH\setminus\{\theta\}$ such that $\eta+\eta'$ is a root, and this root is $\theta$.
It is well known that $\#\gH=2(\rho,\theta^\vee)-1=2h^*-3$, where $h^*$ is the {\it dual Coxeter number\/} 
of $\Delta$. Since $\#I(\ap)_{\sf min}=(\rho,\theta^\vee)=h^*-1$ for $\ap\in\Pi_l$, the ideal 
$I(\ap)_{\sf min}$ contains $\theta$ and exactly a half of elements of $\gH\setminus \{\theta\}$, cf. 
also~\cite[Lemma\,3.3]{jems}.

Although the affine Weyl group and minuscule elements are not 
explicitly used in this paper, their use is hidden in properties of the posets $\Ab_\mu$, $\mu\in\Delta^+_l$, 
and ideals $I(\mu)_{\sf min}$, $I(\mu)_{\sf max}$. Important properties of the maximal abelian ideals are also obtained in \cite{cp3,suter}.

We refer to \cite{bour}, \cite[\S\,3.1]{t41} for standard results on root systems and Weyl groups and 
to~\cite[Chapter\,3]{stan} for posets. 

\section{The greatest lower bound in $\Delta^+$}

\noindent
It is proved in \cite[Appendix]{imrn} that the poset $(\Delta^+,\curge)$ is a join-semilattice. i.e., for any pair 
$\eta,\eta'\in \Delta^+$,  there is the least upper bound (= {\it join}), denoted  $\eta\vee\eta'$. Furthermore, 
there is a simple explicit formula for `$\vee$', see \cite[Theorem~A.1]{imrn}. However, $\Delta^+$ is not a 
meet-semilattice. We prove below that under a natural constraint the greatest lower bound (= {\it meet}) 
exists and can explicitly be described.
Afterwards, we provide some applications of this property in the theory of abelian ideals.

\begin{df}  Let $\eta,\eta'\in \Delta^+$. The root $\nu$ is the greatest lower bound (or {\it meet}) of $\eta$ 
and $\eta'$ if 
\\ \indent 
\textbullet \quad $\eta\curge \nu$, $\eta'\curge\nu$;
\\ \indent 
\textbullet \quad if $\eta\curge \kappa$ and $\eta'\curge \kappa$, then $\nu\curge\kappa$. 
\\
The meet of $\eta$ and $\eta'$, if it exists, is denoted by  $\eta\wedge\eta'$.
\end{df}

Obviously, if $\ap,\ap'\in\Pi$, then their meet does not exist. But as we see below, the only reason for such a failure is that their supports are disjoint.

\begin{lm}[see {\cite[Lemma\,3.1]{adv01}}] Suppose that $\gamma\in\Delta^+$ and $\ap,\beta\in\Pi$.
If $\gamma-\ap, \gamma-\beta\in\Delta^+$, then either $\gamma-\ap-\beta\in \Delta^+$
or $\gamma=\ap+\beta$ and hence $\ap,\beta$ are adjacent in the Dynkin diagram.
\end{lm}

\begin{lm}[see {\cite[Lemma\,3.2]{jlt16}}]  Suppose that $\gamma\in\Delta^+$ and $\ap,\beta\in\Pi$. 
If  $\gamma+\ap, \gamma+\beta\in \Delta^+$, then $\gamma+\ap+\beta\in\Delta^+$.
\end{lm}

Let us provide a reformulation of these lemmata in terms of `$\vee$' and `$\wedge$'. To this end, we note
that in the previous lemma, $(\gamma+\ap)\wedge( \gamma+\beta)=\gamma$.

\begin{prop}    \label{prop:vee-and-wedge}
Let $\eta_1,\eta_2\in \Delta^+$.
\begin{itemize}
\item[{\sf (i)}]  \ If $\eta_1\vee\eta_2$ covers both $\eta_1$ and $\eta_2$, then
either $\eta_1\vee\eta_2=\ap+\beta=\eta_1+\eta_2$ for some adjacent $\ap,\beta\in\Pi$, or
$\eta_1$ and $\eta_2$ both cover $\eta_1\wedge \eta_2$;
\item[{\sf (ii)}] \  If $\eta_1\wedge \eta_2$ exists and $\eta_1$ and $\eta_2$ both cover $\eta_1\wedge \eta_2$, 
then $\eta_1\vee\eta_2$ covers both $\eta_1$ and $\eta_2$.
\end{itemize}
\end{prop}

For any two roots $\eta=\sum_{\ap\in \Pi}c_\ap \ap$ and $\eta'=\sum_{\ap\in \Pi}c'_\ap \ap$, one defines 
two elements of the root lattice, $\min (\eta,\eta')=\sum_{\ap\in \Pi} \min\{c_\ap,c'_\ap\}\ap$
and $\max (\eta,\eta')=\sum_{\ap\in \Pi} \max\{c_\ap,c'_\ap\}\ap$.
Recall that the poset $(\Delta^+,\curge)$ is graded and the rank function is the usual {\it height\/} of a root,
i.e., $\hot(\eta)=\sum_{\ap\in\Pi}c_\ap$. We also set $\hot_\ap(\eta):=c_\ap$.

\begin{thm}   \label{thm:inf=min}   \leavevmode\par
\begin{itemize}
\item[\sf 1)]  \ For any $\gamma\in \Delta^+$, the upper ideal $I\langle{\curge}\gamma\rangle$ is a modular lattice;
\item[\sf 2)]  \ the meet $\gamma_1\wedge\gamma_2$ exists if and only if\/ $\supp(\gamma_1)\cap\supp(\gamma_2)\ne\varnothing$. In this case, one has
$\gamma_1\wedge\gamma_2=\min(\gamma_1,\gamma_2)$.
\end{itemize}
\end{thm}
\begin{proof}
1) By~\cite[Theorem~A.1(i)]{imrn},  the join always exists in $\Delta^+$ and formulae for `$\vee$' show that
$\gamma_1\vee\gamma_2\in I\langle{\curge}\gamma\rangle$ whenever $\gamma_1,\gamma_2\in I\langle{\curge}\gamma\rangle$.
Therefore, $I\langle{\curge}\gamma\rangle$ is a join-semilattice with a unique minimal element. Hence the meet also exists for any $\gamma_1,\gamma_2\in I\langle{\curge}\gamma\rangle$, 
see~\cite[Prop.\,3.3.1]{stan}. That is, $I\langle{\curge}\gamma\rangle$ is a lattice.
Note that, for $\gamma_1,\gamma_2\in I\langle{\curge}\gamma\rangle$, the first possibility in
Proposition~\ref{prop:vee-and-wedge}(i) does not realise.  Therefore, using 
Proposition~\ref{prop:vee-and-wedge} with $I\langle{\curge}\gamma\rangle$ in place of $\Delta^+$ and 
\cite[Prop.\,3.3.2]{stan}, we conclude that $I\langle{\curge}\gamma\rangle$ is a modular lattice.

Yet, this does not provide a formula for the meet and leaves a theoretical possibility that $\gamma_1\wedge\gamma_2$
depends on $\gamma$.

2) If $\supp(\gamma_1)\cap\supp(\gamma_2)=\varnothing$, then there are no roots $\nu$ such that
$\gamma_1\curge \nu$ and $\gamma_2\curge\nu$. Conversely, if
$\supp(\gamma_1)\cap\supp(\gamma_2)\ne\varnothing$, then $\gamma_1,\gamma_2\in 
I\langle{\curge}\gamma\rangle$ for some $\gamma$. Using again~\cite[Prop.\,3.3.2]{stan}, the modularity of the lattice $I\langle{\curge}\gamma\rangle$ implies that
$\hot(\gamma_1\vee\gamma_2)+\hot(\gamma_1\wedge\gamma_2)=\hot(\gamma_1)+
\hot(\gamma_2)$, where $\gamma_1\wedge\gamma_2$ is taken inside $I\langle{\curge}\gamma\rangle$.
It is clear that $\gamma_1\wedge\gamma_2\curle\min(\gamma_1,\gamma_2)$. Moreover, in this situation, 
the formulae of~\cite[Theorem~A.1(i)]{imrn} imply that 
$\gamma_1\vee\gamma_2=\max(\gamma_1,\gamma_2)$. 
Therefore, $\hot(\min(\gamma_1,\gamma_2))=\hot(\gamma_1\wedge\gamma_2)$ and thereby
$\min(\gamma_1,\gamma_2)=\gamma_1\wedge\gamma_2$.
\end{proof}

\begin{rmk}   \label{rem:modular}
A special class of modular lattices inside $\Delta^+$ occurs in connection with $\BZ$-gradings of $\g$. 
For $\ap\in\Pi$, set $\Delta_\ap(i)=\{\gamma\in \Delta \mid \hot_\ap(\gamma)=i\}$. It is known that
$\Delta_\ap(i)$ has a unique minimal and a unique maximal element, see Section~\ref{sect:Z-gr-&-nc}. It is also clear that $\gamma_1\wedge\gamma_2$ and
$\gamma_1\vee\gamma_2\in\Delta_\ap(i)$ for all $\gamma_1, \gamma_2\in\Delta_\ap(i)$. Hence 
$\Delta_\ap(i)$ is a {\bf modular} lattice. (It was already noticed in \cite[Appendix]{imrn} that
$\Delta_\ap(i)$ is a lattice.)
\end{rmk}

\begin{rmk}   \label{rem:theta-fundam}
In what follows, we have to distinguish the $\GR{A}{n}$-case from the other types. One the reasons is that
$\theta$ is not a multiple of a fundamental weight only for $\GR{A}{n}$. In all other types, there is a unique
$\ap_\theta\in\Pi$ such that $(\theta,\ap_\theta)\ne 0$. For the $\BZ$-grading associated with 
$\ap_\theta$, one then has $\Delta_{\ap_\theta}(1)=\gH\setminus \{\theta\}$ and 
$\Delta_{\ap_\theta}(2)= \{\theta\}$. That is, $\gH\setminus \{\theta\}$ (or just $\gH$) has a unique minimal 
element, which is $\ap_\theta$, if and only if $\Delta$ is not of type $\GR{A}{n}$. This provides the following consequence of
Theorem~\ref{thm:inf=min}: \\ \centerline{
{\it If $\Delta$ is not of type $\GR{A}{n}$, then for all
$\eta_1,\eta_2\in \gH\setminus \{\theta\}$, the meet $\eta_1\wedge\eta_2$ exists and lies in
$\gH\setminus \{\theta\}$.}}
\noindent
This is going to be used several times in Section~\ref{sect:subsets-of-Pi_l}.
\end{rmk}

\section{$\BZ$-gradings and non-commutative roots}   \label{sect:Z-gr-&-nc}

\noindent
If $\gamma\in \Delta^+_{\sf com}$, then $\gamma$ belongs to a maximal abelian ideal.
Since $I(\ap)_{\sf max}$, $\ap\in\Pi_l$, are all the maximal abelian ideals in $\Delta^+$, we have
\[
       \Delta^+_{\sf com}=\bigcup_{\ap\in\Pi_l}  I(\ap)_{\sf max} .    
\]
Set $\Delta^+_{\sf nc}=\Delta^+\setminus \Delta^+_{\sf com}$. In this section, we obtain an {\sl a priori} description of $\Delta^+_{\sf nc}$. 
Let us introduce special elements of the root lattice
\beq   \label{eq:thety}
 \thi =\sum_{\ap\in\Pi}  \lfloor\hot_\ap(\theta)/2\rfloor \ap \ \text{ and } \ \tthe=\sum_{\ap\in\Pi}  \lceil\hot_\ap(\theta)/2\rceil \ap .
\eeq
Hence $\thi+\tthe=\theta$. Note that $\lfloor\theta/2\rfloor =0$ if and only if $\theta=\sum_{\ap\in\Pi}\ap$, 
i.e., $\Delta$ is of type $\GR{A}{n}$.

\begin{lm}   \label{lm:nc-&-theta/2}
Suppose that $\Delta$ is not of type $\GR{A}{n}$, so that $\lfloor\theta/2\rfloor \ne 0$. 

{\sf (1)} \ If $\gamma\in \Delta^+_{\sf nc}$, then $\hot_\ap\gamma\le \lfloor\hot_\ap(\theta)/2\rfloor $ for all $\ap\in\Pi$,
i.e., $\gamma\curle \lfloor\theta/2\rfloor $.

{\sf (2)} \ If\/ $\gamma_1,\gamma_2\curle \lfloor\theta/2\rfloor $, then $\gamma_1\vee\gamma_2\curle \lfloor\theta/2\rfloor $.
\end{lm}
\begin{proof}
{\sf (1)} \ Obvious.
\\ \indent
{\sf (2)} \  
By~\cite[Theorem~A.1]{imrn}, if $\supp(\gamma_1)\cup\supp(\gamma_2)$ is connected, then $\gamma_1\vee\gamma_2=
\max(\gamma_1,\gamma_2)$ and the assertion is clear. Otherwise, 
$\gamma_1\vee\gamma_2=\gamma_1+(\text{connecting root})+\gamma_2$. Recall that if the union of supports is not connected, then there is a (unique) chain of simple roots that connects them. If this chain
consists of $\ap_{i_1},\dots,\ap_{i_s}$, then the "connecting root" is $\ap_{i_1}+\dots +\ap_{i_s}$.
Here we only need the condition that $\hot_{\ap}(\theta)\ge 2$ for any $\ap$ in the connecting chain.
Indeed, the roots in this chain are not extreme in the Dynkin diagram, and outside type $\GR{A}{n}$ the coefficients of non-extreme simple roots are always $\ge 2$.
\end{proof}

\begin{rema}
For $\GR{A}{n}$, $\lfloor\theta/2\rfloor = 0$ and hence $\Delta^+_{\sf nc}= \varnothing$.
\end{rema}
Set $\gA=\{\gamma\in \Delta^+\mid \gamma\curle \lfloor\theta/2\rfloor \}$. Then $\gA\ne \varnothing$ if and
only if $\Delta$ is not of type $\GR{A}{n}$. It follows from 
Lemma~\ref{lm:nc-&-theta/2} that 

\textbullet \quad $\Delta^+_{\sf nc}\subset \gA$;

\textbullet \quad  $\gA$ has a unique maximal element.

\noindent
Our goal is to prove that $\Delta^+_{\sf nc}= \gA$ and $\max(\gA)=\{\lfloor\theta/2\rfloor \}$. The latter essentially boils down to the assertion that $\lfloor\theta/2\rfloor $ is a root.

For an arbitrary $\ap\in\Pi$, consider the $\BZ$-grading $\g=\bigoplus_{i\in\BZ}\g_\ap(i)$ corresponding to
$\ap$. That is, the set of roots of $\g_\ap(i)$ is $\Delta_\ap(i)$, see Remark~\ref{rem:modular}. In particular,
$\ap\in \Delta_\ap(1)$ and $\Pi\setminus \{\ap\}\subset \Delta_\ap(0)$. Here $\el:=\g_\ap(0)$ is reductive 
and contains the Cartan subalgebra $\te$. By an old result of Kostant (see \cite{ko10} and Joseph's 
exposition in~\cite[2.1]{jos98}), each $\g_\ap(i)$, $i\ne 0$, is a simple $\el$-module. Therefore, 
$\Delta_\ap(i)$ has a unique minimal and a unique maximal element. 
The following is a particular case of Theorem~2.3 in~\cite{ko10}.

\begin{prop}    \label{prop:non-zero}
If\/ $i+j\le \hot_\ap(\theta)$, then $0\ne [\g_\ap(i),\g_\ap(j)]=\g_\ap(i+j)$.
\end{prop}
Once one has proved that $[\g_\ap(i),\g_\ap(j)]\ne 0$, the equality $[\g_\ap(i),\g_\ap(j)]=\g_\ap(i+j)$ stems from the fact that 
$\g_\ap(i+j)$ is a simple $\el$-module. We derive from this result two corollaries.

\begin{cl}    \label{cl:11}
For any $\mu\in\Delta_\ap(i)$, there is $\nu\in\Delta_\ap(j)$ such that $\mu+\nu\in\Delta_\ap(i+j)$.
\end{cl}
\begin{proof}
Let $e_{\mu}\in \g_\ap(i) $ be a root vector for $\mu$. Assume that the property in question does not hold.
Then  $[e_\mu, \g_\ap(j)]=0$. Hence $[L{\cdot}e_{\mu}, \g_\ap(j)]=0$, where $L\subset G$ is  
the connected reductive group with Lie algebra $\el$. Since the linear span of an $L$-orbit in a 
simple $L$-module is the whole space, this implies that
$[\g_\ap(i),\g_\ap(j)]=0$, which contradicts the proposition.
\end{proof}

Set  $d_{\ap}=\lfloor\hot_\ap(\theta)/2\rfloor $, and let $\mu_{d_{\ap}}$ be the lowest weight in
$\Delta_\ap(d_{\ap})$.

\begin{cl}       \label{cl:2}
$\mu_{d_{\ap}}\in \Delta^+_{\sf nc}$.
\end{cl}
\begin{proof}
By Corollary~\ref{cl:11}, there is $\lb\in \Delta_\ap(d_{\ap})$ such that $\mu_{d_{\ap}}+\lb$ is a root
in $\Delta_\ap(2d_\ap)$.
Since $\mu_{d_{\ap}}\curle\gamma$, the upper ideal in $\Delta^+$ generated by $\mu_{d_{\ap}}$ is 
not abelian.
\end{proof}

This allows us to obtain the promised characterisation of $\Delta^+_{\sf nc}$.

\begin{thm}     \label{thm:delta_nc}
If $\lfloor\theta/2\rfloor \ne 0$, i.e., $\Delta$ is not of type $\GR{A}{n}$, then
$\lfloor\theta/2\rfloor $ is the unique maximal element of $\Delta^+_{\sf nc}$. Furthermore, 
$\lfloor\theta/2\rfloor \in\gH $.
\end{thm}
\begin{proof}
It was noticed above that $\Delta^+_{\sf nc}\subset \gA$, $\gA$ has a unique maximal element, say 
$\hat\nu$, and $\hat\nu\curle \lfloor\theta/2\rfloor $. By Corollary~\ref{cl:2}, for any $\ap\in\Pi$, there is $\mu_\ap\in \Delta^+_{\sf nc}$ such that 
$\hot_\ap(\mu_\ap)=d_{\ap}$. Therefore 
$\underset{\ap\in\Pi}{\bigvee}\mu_\ap\curge \lfloor\theta/2\rfloor $. On the other hand, 
$\mu_\ap\curle \hat\nu$ for each $\ap$ and hence $\underset{\ap\in\Pi}{\bigvee}\mu_\ap \curle \hat\nu\curle \lfloor\theta/2\rfloor $.
Thus, $\lfloor\theta/2\rfloor =\hat\nu$ is a root.
If $\ap_\theta$ is the unique simple root such that $(\theta,\ap_\theta)\ne 0$, then $\hot_{\ap_\theta}(\theta)=2$. Therefore $\lfloor\theta/2\rfloor\in \gH$ whenever $\Delta$ is not $\GR{A}{n}$.
\end{proof}

The fact that $\lfloor\theta/2\rfloor $  is the unique maximal non-commutative root has been observed in 
 \cite[Sect.\,4]{jac06} via a case-by-case analysis.

\begin{ex}
If $\Delta$ is of type $\GR{E}{8}$, then 
$\theta=\ \text{\begin{E8}{2}{3}{4}{5}{6}{4}{2}{3}\end{E8}}$ \ and $\lfloor\theta/2\rfloor =\ \text{\begin{E8}{1}{1}{2}{2}{3}{2}{1}{1}\end{E8}}$ .
\end{ex}

\begin{rmk} In the proof of Corollary~\ref{cl:2} and then Theorem~\ref{thm:delta_nc}, we only need the property, which follows from Proposition~\ref{prop:non-zero}, that $[\g_\ap(d_{\ap}),\g_\ap(d_{\ap})]=\g_\ap(2d_{\ap})$.
\\ \indent 
For $\ap\in\Pi$ with $\hot_\ap(\theta)=2$ or $3$, this means that $[\g_\ap(1),\g_\ap(1)]=\g_\ap(2)$, which 
is obvious. This covers all classical simple Lie algebras, $\GR{E}{6}$, and $\GR{G}{2}$. For 
$\GR{E}{7}$, $\GR{E}{8}$, and $\GR{F}{4}$, there are $\ap\in\Pi$ such that
$\hot_\ap(\theta)\in \{4,5,6\}$. Then the required relation is $[\g_\ap(2),\g_\ap(2)]=\g_\ap(4)$
or $[\g_\ap(3),\g_\ap(3)]=\g_\ap(6)$. This can easily be verified case-by-case.  However, our intention is to provide
a case-free treatment of this  property.
\end{rmk}

Another consequence of Kostant's theory~\cite{ko10} is that one obtains an explicit presentation of some
maximal abelian ideals.

\begin{prop}   \label{prop:max-ab-odd}
Suppose that $\hot_\ap(\theta)=2d_\ap+1$ is odd. Then 
$\ah:=\bigoplus_{j\ge d_\ap+1} \g_\ap(j)$ (i.e., $\Delta_\ah:=\bigcup_{j\ge d_\ap+1}\Delta_\ap(j)$ in the combinatorial set up) is a maximal abelian ideal of\/ $\be$. 
\end{prop}
\begin{proof}
Obviously, $\ah$ is abelian. Let $\lb\in\Delta_\ap(d_\ap)$ be the highest weight. It follows from 
the simplicity of all $\el$-modules $\g_\ap(i)$ that $\lb$ is the only maximal element of $\Delta^+\setminus\Delta_\ah$. Therefore, it suffices to prove that the upper ideal $\Delta_\ah\cup\{\lb\}$ 
is not abelian. Indeed, there is $\nu\in \Delta_\ap(d_\ap+1)$ such that $\nu+\lb$ is a root
(apply Corollary~\ref{cl:11} with $i=d_\ap$ and $j=d_\ap+1$.)
\end{proof}

This prompts the following question. Suppose that $\hot_\ap(\theta)=2d_\ap+1$. Then 
$\ah=I(\beta)_{\sf max}$ for some $\beta \in\Pi_l$. What is the relationship between $\ap$ and $\beta$? 
We say below that $\ap\in\Pi$ is {\it odd}, if $\hot_\ap(\theta)$ is odd.

\begin{ex}   \label{ex:max-ab}
{\sf 1)}  If $\hot_\ap(\theta)=1$, i.e., $d_\ap=0$, then $\ah$ is the (abelian) nilradical of the corresponding 
maximal parabolic subalgebra. Then $\beta=\ap$. This covers all simple roots and all maximal abelian 
ideals in type $\GR{A}{n}$.
\\ \indent
{\sf 2)} For $\Delta$ of type $\GR{D}{n}$ or $\GR{E}{n}$, there are exactly three odd simple roots 
$\ap$. 
\\ \indent
{\bf --} For $\GR{D}{n}$, these are the endpoints of the Dynkin diagram and $d_\ap=0$. That is, again $\ap=\beta$ in these cases.
\\ \indent
{\bf --} For $\GR{E}{n}$, there are also odd simple roots with $d_\ap\ge 1$ and then $\beta\ne\ap$.
\\
Nevertheless, the related maximal abelian ideals always correspond to the extreme nodes of the Dynkin 
diagram! Moreover, one always has $\hot_\beta(\theta)=d_\ap+1$. (Similar things happen for 
$\GR{F}{4}$ and $\GR{G}{2}$.) It might be interesting to find a reason behind it.

Below is the table of all exceptional cases with $d_\ap\ge 1$. The numbering of simple roots 
follows~\cite[Tables]{t41}. In particular, the numbering for $\GR{E}{8}$ is 
\raisebox{-1.7ex}{\begin{tikzpicture}[scale= .7, transform shape]
\tikzstyle{every node}=[circle]
\node (h) at (-.3,0) {\bf 1};
\node (a) at (0,0) {\bf 2};
\node (b) at (.3,0) {\bf 3};
\node (c) at (.6,0) {\bf 4};
\node (d) at (.9,0) {\bf 5};
\node (e) at (1.2,0) {\bf 6};
\node (f) at (1.5,0) {\bf 7};
\node (g) at (.9,-.5) {\bf 8};
\end{tikzpicture}} and the extreme nodes correspond to $\ap_1,\ap_7,\ap_8$. 
\end{ex}

\begin{center}
\begin{tabular}{>{$}c<{$}| >{$}c<{$} |>{$}c<{$} >{$}c<{$} |>{$}c<{$} >{$}c<{$} >{$}c<{$}| >{$}c<{$} >{$}c<{$}|} 
 & \GR{E}{6} & \GR{E}{7} & & \GR{E}{8} & & & \GR{F}{4} & \GR{G}{2} \\ \hline \hline
\ap & \ap_3 & \ap_3 & \ap_5 & \ap_2 & \ap_4 & \ap_8 & \ap_3 & \ap_1  \\
d_\ap & 1 & 1 & 1 & 1 & 2 & 1 & 1 & 1 \\
\beta & \ap_6 & \ap_7 & \ap_6 & \ap_1 & \ap_8 & \ap_7 & \ap_4 & \ap_2  \\
\hot_\beta(\theta) & 2 & 2 & 2 & 2 & 3 & 2 & 2 & 2 \\
\end{tabular}                
\end{center}
\section{Bijections related to the maximal abelian ideals}
\label{sect:subsets-of-Pi_l}

\noindent
In  this section, we consider abelian ideals of the form $I(\ap)_{\sf min}$ and $I(\ap)_{\sf max}$ for
$\ap\in\Pi_l$, and their derivatives (intersections and unions).

The following ресулт is Theorem~4.7  in~\cite{jems}.

\begin{thm}    \label{thm:4.7-jems}
For any $\ap\in\Pi_l$, there is a one-to-one correspondence between
$\min\bigl( I(\ap)_{\min}\bigr)$ and $\max \bigl(\Delta^+\setminus I(\ap)_{\max}\bigr)$.
Namely, if $\eta\in 
\max \bigl(\Delta^+\setminus I(\ap)_{\max}\bigr)$, then $\eta':=\theta-\eta \in \min\bigl( I(\ap)_{\min}\bigr)$, 
and vice versa.
\end{thm}

It formally follows from this theorem that $\min\bigl( I(\ap)_{\min}\bigr)$ and $\max \bigl(\Delta^+\setminus I(\ap)_{\max}\bigr)$ both belong to $\gH$. This is clear for the former, since $I(\ap)_{\sf min}\subset \gH$.
And the key point in the proof of Theorem~\ref{thm:4.7-jems} was to demonstrate {\sl a priori} that 
$\max \bigl(\Delta^+\setminus I(\ap)_{\max}\bigr) \subset \gH$.

Below, we provide a generalisation of Theorem~\ref{thm:4.7-jems}, which is even more general than
\cite[Theorem\,4.9]{jems}, i.e., we will {\bf not} assume that $S\subset\Pi_l$ be connected. 
Another improvement is that we give a conceptual proof of that generalisation, while Theorem~4.9 
in \cite{jems} was proved case-by-case and no details has been given there.

The following is a key step for our generalisation of Theorem~\ref{thm:4.7-jems}.

\begin{thm}    \label{thm:main1}
Suppose that $S\subset \Pi_l$ and 
$\gamma\in \displaystyle \max \bigl(\Delta^+\setminus \bigcup_{\ap\in S} I(\ap)_{\max}\bigr)$. 
If $\Delta$ is not of type $\GR{A}{n}$, then $\gamma\in\gH$.
\end{thm}
\begin{proof}
Here we have to distinguish two possibilities: either $\gamma\in\Delta^+_{\sf nc}$ or 
$\gamma\in\Delta^+_{\sf com}$.

{\sf (1)} \ Suppose that $\gamma\in\Delta^+_{\sf nc}$ and assume that $\gamma\not\in\gH$.  Then there are $\eta,\eta'\succ \gamma$ such that $\eta+\eta'=\theta$, see~\cite[p.\,1897]{imrn}. Here both $\eta$ and $\eta'$ belong to
$\gH\cap \bigl(\bigcup_{\ap\in S} I(\ap)_{\max}\bigr)=\bigcup_{\ap\in S} I(\ap)_{\min}$. Since
$\Delta$ is not of type $\GR{A}{n}$, $\gH$ has a unique minimal element (=\, the unique simple root that is not orthogonal to $\theta$). Therefore, $\mu:=\eta\wedge\eta'$ exists and belongs to $\gH$.
(The existence of $\eta\wedge\eta'$ also follows from Theorem~\ref{thm:inf=min}(1).) Since 
$\eta,\eta'\curge \mu$, we have $\mu\in\Delta^+_{\sf nc}$. This implies that
$\mu\not \in \bigcup_{\ap\in S} I(\ap)_{\min}$ and hence 
$\mu\not\in \bigcup_{\ap\in S} I(\ap)_{\max}$. By the definition of meet, $\gamma\curle \mu$.
Furthermore, $\gamma\not\in\gH$ and $\mu\in\gH$. Hence $\gamma\prec \mu$ and $\gamma$ is not maximal in $\Delta^+\setminus \bigcup_{\ap\in S} I(\ap)_{\max}$. A contradiction!

{\sf (2)} \ Suppose that $\gamma\in\Delta^+_{\sf com}$. Consider the abelian ideal 
$J=I\langle{\curge}\gamma\rangle$. By the assumption, $J\setminus\{\gamma\} \subset
\bigcup_{\ap\in S} I(\ap)_{\max}$. On the other hand, since $J\not\subset I(\ap)_{\sf max}$ for each $\ap\in S$, we conclude that
\[
 J\cap \gH\not\subset I(\ap)_{\sf max}\cap\gH=I(\ap)_{\min} , 
\]
see \cite[Prop.\,3.2]{jems}. For each $\ap\in S$, we pick 
$\eta_\ap \in (J\cap\gH)\setminus I(\ap)_{\sf min}$. Then $\eta_\ap\curge \gamma$. Since $\Delta$ is not 
of type $\GR{A}{n}$, the meet 
$\eta:=\underset{\ap\in S}{\wedge}\eta_\ap$ exists and belong to $\gH$ (Remark~\ref{rem:theta-fundam}) 
and also $\eta\curge\gamma$. Note also that $\eta\not\in I(\ap)_{\min}$ for each $\ap\in S$.
(Otherwise, if $\eta\in I(\ap_0)_{\min}$, then $\eta_{\ap_0}\in I(\ap_0)_{\min}$ as well.)
Therefore, $\eta \not \in \bigcup_{\ap\in S} I(\ap)_{\min}$ and hence 
$\eta\not\in \bigcup_{\ap\in S} I(\ap)_{\max}$ (because $\eta\in\gH$). 
As $\gamma$ is assumed to be maximal in $\Delta^+\setminus \bigcup_{\ap\in S} I(\ap)_{\max}$, we must have $\gamma=\eta\in \gH$.
\end{proof}

\begin{rema}
For $\GR{A}{n}$, this  theorem remains true if we add the hypothesis that $S\subset \Pi_l$ is a {\sl
connected} subset in the Dynkin diagram, see also Example~\ref{ex:sl_n}.
\end{rema}
\begin{thm}    \label{thm:main2}
If $S\subset \Pi_l$ is arbitrary and $\Delta$ is not of type $\GR{A}{n}$, then there is the bijection
\[
 \eta\in \min\bigl( \bigcap_{\ap\in S}I(\ap)_{\min}\bigr)   \stackrel{1:1}{\longmapsto} 
 \eta'=\theta-\eta \in \max \bigl(\Delta^+\setminus \bigcup_{\ap\in S} I(\ap)_{\max}\bigr) .
\] 
\end{thm}
\begin{proof}
{\sf (1)} \ Suppose that $\eta\in \min\bigl( \bigcap_{\ap\in S}I(\ap)_{\min}\bigr)$. As $\Delta$ is not of type
$\GR{A}{n}$, there is a unique $\ap_\theta\in\Pi$ such that $(\theta,\ap_\theta)\ne 0$. Then 
$\theta-\ap_\theta\in \gH$  is the only root covered by $\theta$. Therefore,
$\theta-\ap_\theta\in I(\ap)_{\min}$ for all $\ap\in\Pi_l$. Hence
$\eta\ne \theta$ and hence $\eta'=\theta-\eta$ is a root (in $\gH$). Since 
$\eta\in I(\ap)_{\sf min}$, we have $\eta'\not\in I(\ap)_{\sf min}$, see \cite[Lemma\,3.3]{jems}. And this 
holds  for each $\ap\in S$. Hence $\eta'\not\in \bigcup_{\ap\in S}I(\ap)_{\min}$ and thereby
$\eta' \not\in \bigcup_{\ap\in S}I(\ap)_{\max}$.

Assume that $\eta'$ is not maximal in $\Delta^+\setminus \bigcup_{\ap\in S} I(\ap)_{\max}$, i.e.,
$\eta'+\beta\not\in \bigcup_{\ap\in S} I(\ap)_{\max}$ for some $\beta\in\Pi$. Again,
$\eta'\prec \theta-\ap_\theta$, hence $\eta'+\beta\in \gH\setminus\{\theta\}$.   Then
$\theta-(\eta'+\beta)=\eta-\beta\in \gH$ and arguing ``backwards'' we obtain that $\eta-\beta\in  \bigcap_{\ap\in S}I(\ap)_{\min}$,
which contradicts the fact that $\eta$ is minimal.

{\sf (2)} \ By Theorem~\ref{thm:main1}, if $\eta'\in \max \bigl(\Delta^+\setminus \bigcup_{\ap\in S} I(\ap)_{\max}\bigr)$, then $\eta'\in\gH$. Under these circumstances, the previous part of the proof can be reversed.
\end{proof}

\begin{ex}     \label{ex:sl_n}
Suppose that $\Delta$ is of type $\GR{A}{n}$, with the usual numbering of simple roots. 
Then $I(\ap_i)_{\sf max}=I\langle{\curge}\ap_i\rangle$ for all $i$ and
$\gH=I(\ap_1)_{\sf max}\cup I(\ap_n)_{\sf max}$, where
\begin{gather*}
I(\ap_1)_{\sf min}= I(\ap_1)_{\sf max}=\{\esi_1-\esi_2,\dots,\esi_1-\esi_n,\esi_1-\esi_{n+1}=\theta\},
\\  
I(\ap_n)_{\sf min}=
I(\ap_n)_{\sf max}=\{\esi_n-\esi_{n+1},\dots,\esi_2-\esi_{n+1},\esi_1-\esi_{n+1}\}. 
\end{gather*}
If $S=\{\ap_1,\ap_n\}$, then $S$ is not connected for $n\ge 3$, 
$I(\ap_1)_{\sf min}\cap I(\ap_n)_{\sf min}=\{\theta\}$, and 
$\max \bigl(\Delta^+\setminus (I(\ap_1)_{\sf max}\cup I(\ap_n)_{\sf max})\bigr)=\{\esi_2-\esi_n\}$. That is,
Theorems~\ref{thm:main1} and \ref{thm:main2} do not apply here. However, both remain true 
if $S$ is assumed to be connected and $S\ne \Pi$. For instance, suppose that 
$S=\{\ap_i,\ap_{i+1},\dots,\ap_j\}$ with $1<i<j<n$. Then
$\min\bigl( \bigcap_{\ap\in S}I(\ap)_{\min}\bigr)=\{\esi_1-\esi_{j+1}, \esi_i-\esi_{n+1}\}$ and 
$\max \bigl(\Delta^+\setminus \bigcup_{\ap\in S} I(\ap)_{\max}\bigr)=\{\esi_1-\esi_{i}, \esi_{j+1}-\esi_{n+1}\}$.

If $S=\Pi$, then $\bigcap_{\ap\in \Pi}I(\ap)_{\min}=\{\theta\}$ and
$\Delta^+=\bigcup_{\ap\in \Pi}I(\ap)_{\max}$.
\end{ex}

As a by-product of Theorem~\ref{thm:main2}, we derive a property of maximal abelian ideals
outside type $\GR{A}{}$. Given $S\subset\Pi_l$, let $\langle S\rangle$ be the smallest connected subset
of $\Pi_l$ containing $S$.

\begin{thm}    \label{thm:connected-envelope}   Let $S\subset\Pi_l$. Then
\begin{itemize}
\item[\sf (i)] \  $\bigcap_{\ap\in S} I(\ap)_{\sf min}=\bigcap_{\ap\in \langle S\rangle} I(\ap)_{\sf min}$;
\item[\sf (ii)] \  if $\Delta\ne \GR{A}{n}$, then \ 
$\bigcup_{\ap\in S} I(\ap)_{\sf max}=\bigcup_{\ap\in \langle S\rangle} I(\ap)_{\sf max}$.
\end{itemize}
\end{thm}
\begin{proof}
{\sf  (i)}  By \cite[Theorem\,2.1]{jems},  $\underset{\ap\in S}{\bigcap} I(\ap)_{\sf min}=I(\gamma)_{\sf min}$, where
$\gamma=\underset{\ap\in S}{\bigvee}\ap$. It remains to notice that
$\underset{\ap\in S}{\bigvee}\ap=
\sum_{\ap\in \langle S\rangle}\ap=\underset{\ap\in \langle S\rangle}{\bigvee}\ap$.

{\sf  (ii)} This follows from (i) and Theorem~\ref{thm:main2}. Namely, if $\Delta$ is not of type $\GR{A}{n}$, then 
\[
  \max \bigl(\Delta^+\setminus \bigcup_{\ap\in S} I(\ap)_{\max}\bigr)=
  \max \bigl(\Delta^+\setminus \bigcup_{\ap\in \langle S\rangle} I(\ap)_{\max}\bigr) .  
\]
Hence both unions also coincide.
\end{proof} 
The equality $\underset{\ap\in S}{\bigcap} I(\ap)_{\sf min}=I(\underset{\ap\in S}{\bigvee}\ap)_{\sf min}$
has interesting consequences. By~\cite[Prop.\,4.6]{imrn}, the minimal elements of the abelian ideal
$I(\gamma)_{\sf min}$ have the following description:

{\it Let $w_\gamma\in W$  be a unique element of minimal length such that 
$w_\gamma(\theta)=\gamma$. If $\beta\in\Pi$ and $(\beta,\gamma^\vee)=-1$, then 
$w_\gamma^{-1}(\beta+\gamma)=w_\gamma^{-1}(\beta)+\theta \in \min (I(\gamma)_{\sf min})$. 
Conversely, any element of $\min (I(\gamma)_{\sf min})$ is obtained in this way.}

For any $\gamma$ of the form $\underset{\ap\in S}{\bigvee}\ap$, the required simple roots $\beta$ are easily determined, which yields the maximal elements of $\Delta^+\setminus \bigcup_{\ap\in S} I(\ap)_{\max}$. We consider below the particular case in which $S=\Pi_l$.
 
\begin{prop}     \label{prop:ap-br}
Set $|\Pi_l|=\sum_{\ap\in\Pi_l}\ap$. If $|\Pi_l|\ne \theta$, i.e., $\Delta$ is not of type $\GR{A}{n}$, then
there is a unique $\bap\in \Pi$ such that $|\Pi_l|+\bap$ is a root.
More precisely, 

{\bf -- } \ if $\Delta\in \{\GR{D-E}{}\}$, then $\bap$ is the branching point in the Dynkin diagram;

{\bf -- } \ if $\Delta\in \{\GR{B-C-F-G}{}\}$, then $\bap$ is the unique short root that is adjacent to a long
root in the Dynkin diagram.
\\ \indent
In all these cases, $w_{|\Pi_l|}^{-1}(\bap)=-\lfloor\theta/2\rfloor $.
\end{prop}
\begin{proof}
If $S=\Pi_l$, then $\bigcup_{\ap\in \Pi_l} I(\ap)_{\max}=\Delta^+_{\sf com}$. Hence
$\max \bigl(\Delta^+\setminus \bigcup_{\ap\in \Pi_l} I(\ap)_{\max}\bigr)=\{\lfloor\theta/2\rfloor \}$, see Theorem~\ref{thm:delta_nc}.
Therefore, by Theorem~\ref{thm:main2}, the unique minimal element of 
$I(|\Pi_l|)_{\sf min}=\bigcap_{\ap\in \Pi_l} I(\ap)_{\min}$ is $\theta-\lfloor\theta/2\rfloor=:\lceil \theta/2\rceil $.
This means that there is a unique simple root $\bap$ such that $(|\Pi_l|^\vee,\bap)=-1$, i.e., 
$|\Pi_l|+\bap$ is a root.
Since $w_{|\Pi_l|}^{-1}(|\Pi_l|+\bap)=\theta+w_{|\Pi_l|}^{-1}(\bap)=
\theta-\lfloor\theta/2\rfloor $, the last assertion follows.

Clearly, $\bap$ specified in the statement satisfies the condition that $(|\Pi_l|,\bap)<0$.
\end{proof}

The $\GR{A}{n}$-case can partially be included in the $\GR{DE}{}$-picture, if we formally assume that 
$\bap=0$ (because there is no branching point).

\section{On the interval $[\thi, \tthe]$}
\label{sect:interval}

\noindent
In this section, we first assume that $\Delta$ is not of type $\GR{A}{n}$.
Since $\thi\in\gH$, we have $\tthe=\theta-\thi\in\gH$ and also $\thi\curle \tthe$. We
consider the interval between $\thi$ and $\tthe$ in $\Delta^+$. Let $h$ be the Coxeter number of $\Delta$.

\begin{prop}   \label{prop:interval}
Set $\mathfrak J=\{\gamma\in\Delta^+\mid \thi\curle\gamma\curle\tthe \}$.
 
{\bf -- } \ if $\Delta\in \{\GR{D-E}{}\}$, then  $\mathfrak J\simeq \mathbb B^3$ and $\hot(\tthe)=(h/2)+1$;

{\bf -- } \ if $\Delta\in \{\GR{B-C-F-G}{}\}$, then $\mathfrak J$ is a segment and\/ $\hot(\tthe)=h/2$.
\end{prop}
\begin{proof} 
This can be verified case-by-case, but we also provide some {\sl a priori\/} hints. \\
It follows from the definition of $\thi$, see Eq.~\eqref{eq:thety}, that 
\[
   \tthe-\thi=\theta-2\thi=2\tthe-\theta=\sum_{\ap:\ \hot_\ap(\theta) \text{ odd}} \ap,
\] 
the sum of all odd simple roots. 
Let $\co\subset \Pi$ denote the set of odd roots. Then
$\hot(\tthe)-\hot(\thi)=\#\co$. 

\noindent 
\textbullet \ \ In the simply-laced case, 
$(\theta-2\thi, \thi^\vee)=1-4=-3$. Therefore, there are at least three $\ap\in\co$ such that 
$(\ap, \thi^\vee)=-1$, i.e., $\thi+\ap\in\Delta^+$.  
On the other hand, for any $\gamma\in\Delta^+$, there are at most three $\ap\in\Pi$ such that
$\gamma+\ap\in\Delta^+$~\cite[Theorem\,3.1(i)]{joc12}. Thus, there are exactly three odd roots $\ap_i$
such that $\thi+\ap_i\in\Delta^+$.
Actually,
there are only three odd roots in the $\{\GR{D-E}{}\}$-case.
Hence every odd root can be added to $\thi$. 
Likewise, $(2\tthe-\theta,\tthe^\vee)=3$ and
the same three roots can be subtracted from $\tthe$. This yields all six roots strictly between
$\thi$ and $\tthe$. If $\co=\{\beta_1,\beta_2,\beta_3\}$, then $\mathfrak J$ is as follows:
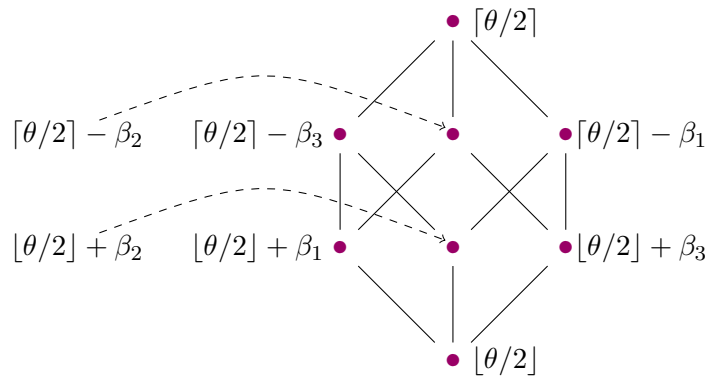
\begin{center}
\begin{tikzpicture}
\node (a) at (0,0) {{\color{redi}$\bullet$}};
\node (h) at (0,4.5) {{\color{redi}$\bullet$}};
\node (b) at (-1.5,1.5) {{\color{redi}$\bullet$}};
\node (c) at (0,1.5) {{\color{redi}$\bullet$}};
\node (d) at (1.5,1.5) {{\color{redi}$\bullet$}};
\node (e) at (-1.5,3) {{\color{redi}$\bullet$}};
\node (f) at (0,3) {{\color{redi}$\bullet$}};
\node (g) at (1.5,3) {{\color{redi}$\bullet$}};
\foreach \from/\to in {a/b, a/c, a/d, b/e, b/f, c/e, c/g, d/f, d/g, e/h, f/h, g/h}  \draw[-] (\from) -- (\to);
\draw (.7,0) node {\small $\thi$}; 
\draw (.7,4.5) node {\small $\tthe$}; 
\draw (-2.6,1.5) node {\small $\thi+\beta_1$}; 
\draw (-5,1.5) node {\small $\thi+\beta_2$}; 
\draw (2.5,1.5) node {\small $\thi+\beta_3$}; 
\draw[dashed, ->]   (-4.7,1.7) .. controls (-2.5,2.5) .. (-.1,1.6);
\draw[dashed, ->]   (-4.7,3.2) .. controls (-2.5,4) .. (-.1,3.1);
\draw (-2.6,3) node {\small $\tthe-\beta_3$}; 
\draw (-5,3) node {\small $\tthe-\beta_2$}; 
\draw (2.5,3) node {\small $\tthe-\beta_1$}; \end{tikzpicture}
\end{center}
\textbullet \ \ In the non-simply laced cases, there is always a unique odd root and hence
$\mathfrak J=\{\thi,\tthe\}$.
\end{proof}

\begin{rmk}  If $\Delta$ is of type $\GR{A}{n}$, then $\thi=0$ and $\tthe=\theta$. Then
$\mathfrak J=\Delta^+\cup\{0\}$. However, this poset is not a modular lattice.

\end{rmk}
\vskip1ex
{\small {\bf Acknowledgements.}
Part of this work was done during my stay at the Max-Planck-Institut f\"ur Mathematik (Bonn). I would like to thank the Institute for its warm hospitality and excellent working conditions.}

\end{document}